\documentclass[10pt,reqno,a4paper]{amsart}
\usepackage[a4paper,left=35mm,right=35mm,top=30mm,bottom=30mm,marginpar=25mm]{geometry}

\usepackage[monochrome]{color}
\usepackage{graphicx}
\usepackage{amsmath,amsfonts,amssymb,amsthm,mathrsfs,amsopn,dsfont,esint}
\usepackage{latexsym}

\usepackage{mathtools}
\usepackage{microtype}

\usepackage{hyperref}

\usepackage[usenames,dvipsnames,x11names,table]{xcolor}

\allowdisplaybreaks

\numberwithin{equation}{section}

\def\R{\mathbb R}

\def\Z{\mathbb Z}
\newcommand{\dist}{\mathop{\mathrm{dist}}}
\def\e{\varepsilon}

\def\spt{{\rm spt}}

\def\00{{\bf 0}}

\DeclareMathOperator*{\esslim}{ess\: lim}
\newcommand\res{\mathop{\hbox{\vrule height 7pt width .3pt depth 0pt \vrule height .3pt width 5pt depth 0pt}}\nolimits}

\def\spt{\mathrm{spt}}
\def\loc{\mathrm{loc}}

\newcommand{\norm}[1]{{\Vert #1\Vert}}
\newcommand{\abs}[1]{{\left\vert #1\right\vert}}
\newcommand{\one}{{\mathds{1}}}

\newtheorem{theorem}{Theorem} [section]
\newtheorem*{theorem*}{Theorem}

\newtheorem{lemma}[theorem]{Lemma}
\newtheorem{proposition}[theorem]{Proposition}
\theoremstyle{definition}
\newtheorem{definition}[theorem]{Definition}
\newtheorem{remark}[theorem]{Remark}
\newtheorem*{ack}{Acknowledgements}

\title[Eikonal]{Optimal Besov differentiability for entropy solutions of the Eikonal equation}

\author[F. Ghiraldin]{Francesco Ghiraldin}
\address{F.G.: Universit\"at Basel, Spiegelgasse 1, CH-4051 Basel, Switzerland}
\email{Francesco.Ghiraldin@unibas.ch}

\author[X. Lamy]{Xavier Lamy}
\address{X.L.:  Institut de Math\'ematiques de Toulouse, CNRS UMR 5219, Universit\'e Paul Sabatier, Toulouse, France}
\email{Xavier.Lamy@math.univ-toulouse.fr}

\begin{document}

\begin{abstract}
In this paper we study the Eikonal equation in a bounded planar domain. We prove the equivalence among optimal Besov regularity, the finiteness of 
every entropy production and the validity of a kinetic formulation.
\end{abstract}
\maketitle

\section{Introduction}

In this paper we study the Eikonal equation in a planar domain: 
given a bounded Lipschitz domain $\Omega\subset\R^2$    
 we consider solutions  $m:\Omega\rightarrow\R^2$ to the following constrained equation
\begin{equation}\label{eq:m}\tag{M}
|m|=1\quad\text{ a.e. in }\Omega, \qquad\nabla\cdot m = 0\quad \text{ in }\mathcal D '(\Omega).
\end{equation}
The eikonal equation \eqref{eq:m} is very flexible, and uniqueness or regularity cannot be expected for such weak solutions, 
even imposing a boundary datum (the equation, on simply connected domains, 
is equivalent to solving $|{  \nabla} u| = 1$ for a scalar function $u$, for which for instance the theory of 
viscosity solutions singles out a distinguished subclass of solutions). 

On the other hand, solutions to \eqref{eq:m} coming from physical models usually possess
 extra information that limit this flexibility. 
 This equation emerges in the description of several physical phenomena, collectively called line-energy  
Ginzburg-Landau models, that describe for instance smectic liquid crystals, soft ferromagnetic films, blister formations, 
and broadly speaking phase transition phenomena where the order parameter is a gradient \cite{DMKO01}. 
 
 For example one can consider the Aviles-Giga energy 
\begin{equation}\label{eq:ag}\tag{AG}
AG_\e(u_\e,\Omega):=\int_\Omega \e |{ \nabla}^2 u_\e|^2 + \frac{1}{\e}(1-|{ \nabla} u_\e|^2)^2 dx, \qquad\Omega\subset\R^d
\end{equation}
(with appropriate boundary conditions): this energy has been introduced in \cite{avilesgiga86} to study liquid crystal configurations, and 
in the two dimensional case was considered by Gioia and Ortiz as a model energy for the deformation of thin film blisters undergoing biaxial compression \cite{gioiaortiz94}. 
The functional \eqref{eq:ag} can be thought of as a vectorial Modica Mortola energy, where the fields are forced to be gradients; equivalently, $\e AG_\e$ can be seen 
as a singular perturbation (accounting for the bending energy of the film) of the elastic energy. Competition between these two terms favors 
concentration along the jump discontinuities of the limit gradient $\nabla u$, 
with a limit energy believed to be asymptotically 
\[
\frac 13 \int_{Jump(\nabla u)} \abs{\nabla^+ u - \nabla^- u}^3 d\mathcal H^1.
\]
Solutions to \eqref{eq:m} can be obtained from sequences $(u_\e)$ with equibounded energy $AG_\e(u_\e)\leq E$, 
by setting 
\begin{equation}\label{eq:limit}
m_\e :={  \nabla}^\perp u_\e,\qquad m_\e\rightarrow m,
\end{equation}
and observing that any pointwise limit $m$ satisfies 
the unitary constraint $|m|=1$ as well as the linear constraint $0 =\nabla\cdot{  \nabla}^\perp u_\e\rightarrow{ \nabla}\cdot m$ 
(see \cite{ambdelman99,DMKO01} for the precise compactness statements).

In \cite{jinkohn00}, Jin and Kohn studied the energy $AG_\e$ and its variations (under suitably boundary conditions), 
and discovered that the divergences ${  \nabla}\cdot \Sigma({ \nabla} u_\e)$ of 
suitable vectorial renormalizations of the gradient fields ${ \nabla} u_\e$ 
are measures providing nontrivial asymptotic  lower bounds for $AG_\e(u_\e)$. 
The explicit form of $\Gamma-\lim_\e AG_\e$ and its domain have been subject to intensive study, see \cite{ambdelman99,avilesgiga99,DMKO01,contidelellis07}, 
where partial results on the $\Gamma-\liminf$ already conjectured in \cite{avilesgiga86,jinkohn00} have been obtained. 

Unit vector fields $m$ obtained through the limit procedure \eqref{eq:limit} enjoy further regularity properties: they are entropy solutions. 
After recognizing that \eqref{eq:m} can be interpreted as a perturbation of Burgers' equation, 
in  \cite{DMKO01} the parallel between these vectorial renormalizations of the eikonal equation 
and entropy solutions of scalar conservation laws was pushed forward, 
and a family $ENT$ of entropies { $\Phi \colon S^1\rightarrow \R^2$} such that ${\nabla}\cdot \Phi(m)$ detect the singularities of $m$, 
has been singled out. It became therefore natural to study {\em entropy solutions} to \eqref{eq:m}, namely 
vector fields satisfying the further property that  
\begin{equation*}
{  \nabla}\cdot\Phi(m)\in\mathcal M(\Omega)\quad\forall\Phi\in ENT,
\end{equation*}
where $\mathcal M(\Omega)$ is the set of finite Radon measures on $\Omega$.
As in the case of hyperbolic conservation laws, such additional informations imply further regularity and compactness of the set of solutions,
in the spirit of Tartar's compensated compactness \cite{tartar79}. 

A quantitative statement of compactness, in the form of fractional differentiability was afterwards obtained, 
only for solutions of \eqref{eq:m} specifically arising as limits of ${ \nabla}^\perp u_\e$ \eqref{eq:limit}, by Jabin and Perthame in \cite{jabinperthame01}. They prove that such solutions satisfy a kinetic formulation:
the equilibrium function (Maxwellian) 
\begin{equation*}
\chi(x,\xi)=\one_{m(x)\cdot\xi>0},
\end{equation*}
defined for $\xi\in\R^2$, solves a transport equation of the form
\begin{equation}\label{eq:kinJP}
\xi\cdot\nabla_x\chi(x,\xi)=\partial_{\xi_1}\sigma_1 + \partial_{\xi_2}\sigma_2 +\sigma_3,
\end{equation}
for some locally finite measures $\sigma_\ell$, see
 \cite[Theorem 1.1]{jabinperthame01}. 
Recall that in the realm of scalar conservation law, the validity of a kinetic formulation is equivalent to the finiteness of all entropy productions \cite{perthame02}.  
With the help of methods coming from velocity averaging, the authors of \cite{jabinperthame01} are able to prove that such 
solutions possess some fractional differentiability: $m \in W^{\frac 15 -,\frac 53 -}_{\rm loc}$; 
better Sobolev regularity $W^{\frac 13 -,\frac 32 -}_{\rm loc}$ was established by the same authors in a subsequent work \cite{jabinperthame02}. 
Examples by De Lellis and Westdickenberg  \cite{delelliswestdickenberg03} show that this regularity is optimal in the number of derivatives ($1/3$) but leave room for improving the integrability.

Similar results hold for weak solutions of Burgers' equation
$\partial_t u +\partial_x \frac 12 u^2 =0$ whose entropy productions are finite measures (but may change sign). This should come as no surprise since Burgers' equation formally arises when considering solutions of \eqref{eq:m} which are small perturbations of the constant solution $m_0=(1,0)$ (see e.g. the discussion in \cite[p.143]{ottosteiner10}). In the case of Burgers' equation, solutions with finite entropy production are shown by Golse and Perthame \cite{golseperthame13} to lie in $B^{1/3}_{3,\infty}$, which is the optimal regularity according to \cite{delelliswestdickenberg03}.

We wish to mention also a similar model arising in the theory of micromagnetism and studied by Rivi\`ere and Serfaty in \cite{riviereserfaty01,riviereserfaty03}. There, solutions of \eqref{eq:m} also appear as limits of sequences with bounded energy depending on a parameter $\e$, and they enjoy a kinetic formulation.
In that model the unit constraint is imposed already at the $\e$ level, thus enforcing a topological restriction, while the divergence free condition
 is only reached in the limit, via the penalization of a nonlocal term. 
This feature makes the limit problem quite different from ours (motivated by the Aviles-Giga functional): there, the field $m_\e$ possesses an $H^1$ lifting $e^{i\varphi}$ (excluding vortices at the $\e$ level - Bloch lines), that enables the use of a convenient family of entropies which control jumps of the angle $\varphi$.  For this model, a quite thorough study of the rectifiability properties of entropy solutions has been carried out in  \cite{AKLR02}. Similar rectifiability properties were then obtained for the present model \cite{delellisotto03} and for higher dimensional scalar conservation laws \cite{delellisottowestdickenberg03}.
An interesting and more sophisticated model describing (almost horizontal) 
micromagnetism in three dimensions, exhibiting different types of
transition layers (one and two dimensional N\'eel walls and Bloch lines) has been considered in \cite{alougesriviereserfaty02}.

Another distinguished subset of solutions to \eqref{eq:m} are 
the so-called zero energy states, for which the field $m$ is again as in \eqref{eq:limit}  
 with the additional property that $\lim_\e AG_\e(u_\e)= 0$. 
Such solutions have no entropy production: ${  \nabla}\cdot\Phi(m)=0$ for all $\Phi\in ENT$. This yields stronger regularity and rigidity properties, as shown by Jabin, Otto and Perthame \cite{jabinottoperthame02}:  $m$ is locally Lipschitz outside 
a locally finite set of points (the {\em vortices}, that asymptotically carry no energy), and in any convex neighborhood of one of them (say $p$), 
it holds $m(x) =\pm\frac{(x-p)^\perp}{|x-p|}$ (see also \cite{bochardignat17} for similar results in higher dimensions). Recently Lorent and Peng \cite{lorentpeng17} showed that the vanishing of only two particular entropy productions (instead of all $\Phi\in ENT$) is needed to obtain this conclusion. 
An indication on the  minimal regularity of $m$ needed to trigger such an improvement 
was further studied by De Lellis and Ignat in \cite{delellisignat15}, where it is proved that if $m \in W^{\frac 13,3}_{{\rm loc}}$ then 
there is no entropy production. 
The $\frac 13$ differentiability exponent is somehow critical in several problems, notably 
the problem of energy conservation for the Euler equations (Onsager conjecture) \cite{constantinetiti94,CCFS08,delellisszekelyhidisurvey}.

In this article we prove the following (see Theorem~\ref{t:main} and Section~\ref{s:main} for the precise definitions): 
\begin{theorem*}
Let $m$ satisfy \eqref{eq:m}. The following three conditions are equivalent:
\begin{itemize}
\item[(i)] $m$ has locally finite entropy production;
\item[(ii)] $m$ satisfies a kinetic formulation;
\item[(iii)] $m\in B^{1/3}_{3,\infty, \loc}(\Omega)$.
\end{itemize}
\end{theorem*}

This Theorem improves the previous literature in several aspects:
the kinetic formulation is deduced from the mere knowledge that all entropy productions are finite, instead of the stronger requirement that $m$ be the limit of an Aviles-Giga sequence \eqref{eq:limit}. Whether or not the latter is \emph{strictly} stronger  is a  nontrivial and, to the authors' knowledge, open question (related to the upper $\Gamma$-limit of $AG_\e$ -- what can be checked by estimating the energy of a convolution is that maps $m\in B^{1/2}_{2,\infty}\cap B^{1/4}_{4,\infty}\subsetneq B^{1/3}_{3,\infty}$ are limits of Aviles-Giga sequences, see \cite{poliakovsky17}). Moreover our kinetic formulation (see \eqref{eq:kin} below) takes a simpler form than \eqref{eq:kinJP}.

The fractional differentiability $B^{1/3}_{3,\infty}$ that we deduce from the kinetic formulation entails improved integrability compared to the previous known one \cite{jabinottoperthame02}. As already mentioned, the corresponding result for Burgers' equation is due to Golse and Perthame \cite{golseperthame13}. 
Their proof relies on a kinetic formulation in which the equilibrium function $\chi$ satisfies some monotonicity assumption. 
This monotonicity is not present in our case, which requires substantial modification of their method.

This fractional differentiability is necessary \emph{and sufficient} (hence optimal). Moreover our calculations also show that slightly better summability (e.g. $m\in B^{1/3}_{3,q}$ for some $q<\infty$, see \S~\ref{s:cor}) already triggers the aforementioned enhanced regularity ($m$ locally Lipschitz outside a discrete set). This criticality of $B^{1/3}_{3,\infty}$ is due to the commutator estimates employed in the argument: similarly to the case of Euler equations, ``energy conservation'' for functions with slightly better differentiability properties can be proved  \cite{delellisignat15,constantinetiti94}.

The proof of the Besov regularity from the kinetic formulation (implication (ii)$\Rightarrow$(iii) in the Theorem) employs an {\em interaction estimate} due to Varadhan \cite{tartar08}, that was used in \cite{golseperthame13} 
and in \cite{goldmanjosienotto15}: as in those works we build a quantity $\Delta(x,z)$ which depends on the equilibrium function
$ \chi(x,\xi)$ and which controls the cubic increment $|m(x+z)-m(x)|^3$. The above-mentioned interaction estimate, together with the kinetic 
formulation, provides an upper bound on $\int_\Omega\Delta(x,z)dx$ in terms of $|z|$, hence the Besov regularity. 
To prove (i)$\Rightarrow$(ii), i.e. the validity of a kinetic formulation from the knowledge of having finite entropy production, we employ a Banach-Steinhaus argument 
as  in \cite{delellisottowestdickenberg03,BBMN10}. The other implication (iii)$\Rightarrow$(i) follows from a careful integration by compensation inspired by \cite{delellisignat15}.

After proving the above Theorem, we explore several questions that come up naturally.
As already mentioned, in the model studied by Rivi\`ere and Serfaty in \cite{riviereserfaty01,riviereserfaty03}, the solutions of \eqref{eq:m} that arise can be written as $m=e^{i\varphi}$ with some control on the lifting $\varphi$. Analogues of our entropy productions and kinetic formulation play a crucial role, and the kinetic defect measure (which is linked to the kinetic formulation) provides a \emph{sharp} lower bound for the energy \cite{riviereserfaty03}.
We show that the corresponding property is \emph{not} present in our case.

Another natural question regards the set of entropies necessary to obtain the Besov regularity. Lorent and Peng prove in \cite{lorentpeng17} that the vanishing of only two particular entropy productions is enough to force all the entropy productions to vanish: is the mere finiteness of these two particular entropy productions also enough to ensure the optimal regularity? We are unable to fully answer this question, but adapting some arguments in \cite{golse10} we do obtain some (lower) regularity.

The article is structured as follows: after some preliminary notations in the next Section~\ref{s:not}, in Section~\ref{s:main} we state and prove the main Theorem, and in Section~\ref{s:cor} we gather some results related to the zero energy states 
and to the above mentioned further natural questions.

\begin{ack}
The authors thank Pierre Bochard for many enlightening discussions and his participation to the proof of Proposition~\ref{p:3}.
F.G. is supported by the ERC Starting Grant 676675 {\em FLIRT - Fluid Flows and Irregular Transport}. 
F.G. acknowledges the kind hospitality of Institut de Math\'ematique de Toulouse, where part of this paper was written. 
\end{ack}

\section{Notations and statement of the problem}\label{s:not}

Given $\xi,\eta\in\R^2$ we identify them with complex number in order to define their scalar and vector products:
\begin{equation*}
\bar\xi\eta = \xi\cdot\eta + i\xi\wedge\eta.
\end{equation*}
Equivalently: $\xi\cdot\eta = \xi_1\eta_1 +\xi_2\eta_2$, $\xi\wedge\eta = \xi_1\eta_2-\xi_2\eta_1$. 
For the sake of clarity we will not always identify unit complex numbers with rotations of the plane: in such occasions, 
rotations by an angle $\theta$ will be denoted by $R_\theta\in SO(2)$.

We will measure the smoothness of our unit fields $m$ in the scale of Besov spaces on a domain: 
in order to keep the notation light, we give the definition only for the exponents we need; 
we refer the reader to \cite{triebel06} for an overview of the definitions; see also \cite{bahourichemindanchin}. 
If $f:\Omega\rightarrow\R$ and $z\in\R^n$, we define $D^z f(x) $ to be the increment $f(x+z)-f(x)$ if both $x,x+z\in\Omega$, and zero otherwise. 
\begin{definition}[Besov spaces on domains, {\cite[Theorem 1.118]{triebel06}}]
Let $\Omega$ be a bounded Lipschitz domain and $1\leq q \leq \infty$. 
A function $f:\Omega\rightarrow\R$ belongs to $B^{\frac 13}_{3,q} (\Omega)$ if 
\[
\|f\|_{L^3(\Omega)} + 	
\left( 
\int_0^1 \left( t^{-\frac 13}
\sup_{|z|\leq t} \| D^zf  \|_{L^3(\Omega)}\right)^q
\frac{dt}{t}
\right)^\frac{1}{q}
<\infty.
\]
The expression on the left hand side provides a quasi-norm. Moreover we denote with $B^{\frac 13}_{3,q,\loc} (\Omega)= \bigcap_{U\subset\subset\Omega} B^{\frac 13}_{3,q} (U)$. 
\end{definition}
It will be convenient to denote, for $U\subset \Omega$, 
\[
N_t(f,U) := \sup_{|z| \leq t} t^{-\frac{1}{3}} \| D^{z} f\|_{L^3(U)}:
\]
with this quantity we can also define the distinguished subspace, lying between $B^{\frac 13}_{3,q} $ with $q<\infty$ and 
$B^{\frac 13}_{3,\infty} $:
\[
B^{\frac 13}_{3,c_0} (\Omega) = B^{\frac 13}_{3,\infty} (\Omega)\cap \{ f : \lim_{t\rightarrow 0} N_t(f,\Omega) = 0\}. 
\]

In order to detect and describe the singularities of solutions of equation~\eqref{eq:m}, it is customary to test the 
equation on suitable renormalization of the solution (see \cite{DMKO01,ignatmerlet12,delellisignat15}):
\begin{definition} 
A function  $\Phi  \in C^2(S^1,\R^2)$ is an entropy for the equation \eqref{eq:m} if  
\[
e^{it}\cdot\frac{d}{dt}\left[\Phi(e^{it})\right]=0\qquad\forall t\in\R.
\]
The set of all entropies is denoted by $ENT$.
\end{definition}

This definition is designed so that any smooth unit field $m$ solving \eqref{eq:m} satisfies ${  \nabla}\cdot\Phi(m) =0$ for $\Phi\in ENT$. In contrast, if $m$ has only bounded variation, $\nabla\cdot\Phi(m)$ will be a measure concentrated on the jump set of $m$, called the entropy production associated to $\Phi$.
In other words, $BV$-type jump discontinuities of $m$ are detected by such divergences: already in \cite{jinkohn00}, in the context of the Aviles-Giga functional~\eqref{eq:ag},  
a special family of ``cubic'' entropies 
$\Sigma_{\alpha_1,\alpha_2}$ were introduced, depending on a chosen orthonormal frame 
of coordinates $(\alpha_1,\alpha_2)$:
\begin{equation}\label{eq:jk}
\Sigma_{\alpha_1,\alpha_2}(z) = \frac 43\left((z\cdot\alpha_2)^3\alpha_1 + (z\cdot\alpha_1)^3\alpha_2\right). 
\end{equation}

The maps $\Sigma_{\alpha_1,\alpha_2}$ are easily seen to belong to $ENT$. 
The divergences  ${ \nabla}\cdot\Sigma_{\alpha_1,\alpha_2}(m)={  \nabla}\cdot\Sigma_{\alpha_1,\alpha_2}(\nabla^\perp u)$ of these entropies detect the jump discontinuities of $\nabla u$, according to the relative orientation of the discontinuity set 
$J_{\nabla u}$ with respect to the chosen frame $(\alpha_1,\alpha_2)$. 
An optimization procedure over the frame bundle provides the lower bound
\begin{equation}\label{eq:cubic}
\Gamma-\liminf_\e AG_\e (u) \geq  \frac{1}{3}\int_{J_{\nabla u}}\abs{\nabla^+ u - \nabla^- u}^3 d\mathcal H^1,
\end{equation}
(in the $W^{1,3}$  topology) at functions $u$ such that $\nabla u\in BV(\Omega)$ and $\abs{\nabla u}=1$ almost everywhere. Here 
$J_{\nabla u}$ is the jump set of the gradient and $\nabla^\pm u$ are its traces on $J_{\nabla u}$, see \cite{ambdelman99}. 
The cubic power of the jump appearing in \eqref{eq:cubic} hints at the Besov scale $B^{\frac 13}_{3,q}$ we are considering here. For functions $u$ with $\nabla u\in BV$, the right-hand side of \eqref{eq:cubic} can be conveniently expressed in terms of the entropy productions, since it holds
\begin{align*}
\frac 13 \abs{\nabla^+ u-\nabla^- u}^3\mathcal H^1\res J_{\nabla u} &= \bigvee_{(\alpha_1,\alpha_2)}\norm{{ \nabla}\cdot\Sigma_{\alpha_1,\alpha_2}(u)}.
\end{align*}
This is proved in \cite[Theorem 3.8]{ambdelman99} (see also \cite{ignatmerlet12}). 
Here $\norm{\mu}$ denotes the total variation measure of a complex-valued measure $\mu$, and the symbol $\bigvee$ denotes the least upper bound of a family of measures \cite[Definition~1.68]{AFP00}: 
\[
\bigvee_{\alpha\in A}\mu_\alpha(E) :=\sup\left\{
\sum_{\{\alpha'\}\subset A} \mu_{\alpha'}(E_{\alpha'}) : \{E_{\alpha'}\} \text{ pairwise disjoint, }E=\bigcup_{\alpha'} E_{\alpha'}
\right\}. 
\]

Hence the estimate \eqref{eq:cubic} provides a control of the entropy production associated to the cubic entropies \eqref{eq:jk} by the Aviles-Giga energy. In fact for any entropy $\Phi\in ENT$ it is shown in \cite{DMKO01} that limits $m$ of sequences $m_\e=\nabla^\perp u_\e$ with $AG_\e(u_\e)\leq M$ satisfy
\begin{equation*}
\norm{\nabla\cdot\Phi(m)}(\Omega)\leq C \norm{D^2\Phi}_\infty \liminf_{\e\to 0} AG_\e(u_\e).
\end{equation*}
In particular all the entropy productions are finite measures. This motivates the following
\begin{definition}
We say that a vector field $m$ solving \eqref{eq:m} has {\em locally finite weak  entropy production} in $\Omega$ if for every $\Phi\in ENT$ 
we have 
\begin{equation}\label{eq:wfep}\tag{wFEP}
{  \nabla}\cdot\Phi(m) \in \mathcal M_{\loc} (\Omega).
\end{equation}
If furthermore  
\begin{equation}\label{eq:fep}\tag{sFEP}
\bigvee_{\Phi\in ENT,\, \|D^2\Phi\|_\infty\leq 1} \|{\nabla}\cdot\Phi(m)\|\in\mathcal M_{\loc}(\Omega),
\end{equation}
 we say that $m$ has {\em locally finite strong  entropy production} in $\Omega$.
\end{definition}

\begin{remark}
Limits of sequences $m_\e=\nabla^\perp u_\e$ with $AG_\e(u_\e)\leq M$ satisfy \eqref{eq:fep}.
\end{remark}

\begin{definition}
We say that a vector field $m$ solving \eqref{eq:m} satisfies the kinetic formulation 
if there exists a Radon measure $\sigma\in\mathcal M_{\loc}( \Omega \times \R/2\pi\mathbb Z )$
such that
\begin{equation}\label{eq:kin}\tag{KIN}
e^{is}\cdot	\nabla_x\one_{e^{is}\cdot	m(x)>0} = \partial_s\sigma \qquad\text{ in }\mathcal D '( \Omega \times \R/2\pi\mathbb Z ).
\end{equation}
\end{definition}

The main Theorem of this paper is the following: 
\begin{theorem}\label{t:main}
Let $m$ satisfy \eqref{eq:m}. The following four conditions are equivalent:
\begin{itemize}
\item[(i)] $m$ has  locally  finite weak entropy production, \eqref{eq:wfep};
\item[(ii)] $m$ satisfies the kinetic equation \eqref{eq:kin};
\item[(iii)] $m\in B^{1/3}_{3,\infty,\loc }(\Omega)$;
\item[(iv)] $m$ has locally  finite strong entropy production, \eqref{eq:fep}.
\end{itemize}
\end{theorem}

\begin{remark}
It is interesting to recall the following boundary behaviour of solutions of conservation laws with finite entropy production \cite[Theorem 1.1]{vasseur01} (see also \cite[Theorem 2.5]{CDDG16}): 
the field $m$ admit a strong $L^1$ trace on $\partial\Omega$, in the sense that there exists a function $v\in L^\infty(\partial\Omega,S^1)$ such that
\[
\esslim_{s\rightarrow 0}\int_{\partial\Omega}|u(\psi(s,x))-v(x)|d\mathcal H^1(x) =0, 
\]
where $\psi$ is a suitable parametrization of a neighborhood of $\partial\Omega$. 
\end{remark}

\section{Proof of the Main Theorem}\label{s:main}
The proof of Theorem \ref{t:main} is divided into three propositions. The implication (iv)$\Rightarrow$(i) is trivial. 

\subsection{Finite entropy implies kinetic formulation}

\begin{proposition}\label{p:1}
If $m$ has weak finite entropy production \eqref{eq:wfep}, then it satisfies the kinetic formulation \eqref{eq:kin}.
\end{proposition}

We will need to construct a suitable family of entropies $\Phi_f$ parametrized (linearly) by continuous functions on $S^1$: 
\begin{equation*}
C^0(\R/2\pi\Z,\R)\ni f\mapsto \Phi_f \in ENT.
\end{equation*}
The construction is done in several steps. First define $\tilde f\in C^0(\R/2\pi\Z,\R)$ by removing the null and the first Fourier modes:
\begin{equation*}
\tilde f(t)=f(t)-\left(\frac{1}{2\pi}\int_0^{2\pi} f(s)\,ds \right) - \left(\frac{1}{\pi}\int_0^{2\pi} f(s)\cos s\, ds\right)  \cos t -\left(\frac{1}{\pi}\int_0^{2\pi} f(s)\sin s\, ds\right)  \sin t.
\end{equation*}
Then define $\psi_f\in C^1(\R/2\pi\Z,\R)$ by
\begin{equation*}
\psi_f(t)=\int_0^t \tilde f(s)\, ds.
\end{equation*}
Note that $\psi_f$ is $2\pi$-periodic since $\int_0^{2\pi}\tilde f=0$. Moreover it holds
\begin{equation*}
\int_0^{2\pi} \psi_f(s) e^{is} \,ds =0,
\end{equation*}
since $\int_0^{2\pi} \tilde f(s)e^{is}\, ds=0$. This allows us to define $\varphi_f\in C^2(\R/2\pi\Z,\R^2)$ by
\begin{equation*}
\varphi_f(t)=\int_0^t \psi_f(s) i e^{is}\, ds.
\end{equation*}
Finally define $\Phi_f\in C^2(\mathbb S^1,\R^2)$ by
\begin{equation*}
\Phi_f(e^{it})=-i\varphi_f(t-\pi/2)+i\varphi_f(t+\pi/2).
\end{equation*}
Then it holds
\begin{align*}
e^{it}\cdot\frac{d}{dt}\left[\Phi_f(e^{it})\right]&=e^{it}\cdot\left(-i\psi_f(t-\pi/2)ie^{i(t-\pi/2)}+i\psi_f(t+\pi/2)ie^{i(t+\pi/2)} \right)\\
&=-\left(\psi_f(t-\pi/2)+\psi_f(t+\pi/2)\right)\, e^{it}\cdot (ie^{it}) =0,
\end{align*}
so that  $\Phi_f\in ENT$. Note that the map $f\mapsto\Phi_f$ is linear, and that $\norm{\Phi_f}_{C^2}\leq C\norm{f}_{C^0}$ for some constant $C>0$.

\begin{remark}\label{r:frames}
Note that $\Phi_{\cos(2t)} = -\frac 12 \Sigma_{e_1,e_2}$ and that $\Phi_{\sin(2t)} = - \frac 12 \Sigma_{\e_1,\e_2}$, where
$(e_1,e_2)$ is the standard basis and $({\e_1,\e_2})$ is its rotation by $\pi/4$. In particular, the classical entropies for the Aviles-Giga functional discovered by Jin and Kohn 
are parametrized by the first nontrivial modes of $f$ (those with wavenumber $2$). 
\end{remark}

The reason for defining the family of entropies $\{\Phi_f\}$ as above lies in its connection to the left-hand side of the kinetic formulation:
\begin{lemma}\label{l:linknuPhif}
Let $\nu:=e^{it}\cdot\nabla_x\left(\one_{e^{it}\cdot m(x)>0}\right)\in \mathcal D'(\Omega\times\R/2\pi\Z)$.
For any $f\in C^\infty(\R/2\pi\Z,\R)$ and $\zeta\in C_c^\infty(\Omega)$ it holds
\begin{equation*}
\langle \nu,\zeta\otimes  \psi_f \rangle = -\langle\nabla\cdot\Phi_f(m),\zeta\rangle.
\end{equation*}
\end{lemma}
\begin{proof}
We have
\begin{align*}
\langle \nu,\psi_f(t)\zeta(x)\rangle =-\int_\Omega\nabla\zeta(x)\cdot\int_{\R/2\pi\Z}\psi_f(t)\one_{e^{it}\cdot m>0}e^{it}\, dt\, dx,
\end{align*}
and for all $x\in\Omega$, writing $m(x)=e^{i\alpha}$ we compute
\begin{align*}
\int_{\R/2\pi\Z}\psi_f(t)\one_{e^{it}\cdot m>0}e^{it}\, dt
&=\int_{\alpha-\pi/2}^{\alpha+\pi/2}\psi_f(t)e^{it}\,dt\\
&=\int_{\alpha-\pi/2}^{\alpha+\pi/2}\frac{1}{i}\varphi_f'(t)\,dt\\
&=-i\varphi_f(\alpha+\pi/2)+i\varphi_f(\alpha-\pi/2)\\
&=\Phi_f(m),
\end{align*}
hence $\langle \nu,\zeta\otimes  \psi_f \rangle =- \int_\Omega\Phi_f(m)\cdot\nabla\zeta\, dx$.
\end{proof}

The next lemma provides the measure $\sigma$ appearing in the right-hand side of the kinetic formulation: 
as in \cite[Theorem 3.1.6]{Perthame}, the entropy production of the solution $u$ of a conservation law under a certain entropy $S$ can be written as an integral 
of $S''$ against the so-called entropy measure. In our case, observe that $\Phi_f$ is obtained by integrating  $f$ twice. 
\begin{lemma}\label{l:sigma}
If $m$ has  locally finite weak entropy production in $\Omega$, then there exists $\sigma \in \mathcal{M}_{\loc}(\Omega\times \R/2\pi\mathbb Z)$ satisfying
\begin{equation}\label{eq:sigma}
\langle\nabla\cdot\Phi_f(m), \zeta\rangle = \iint_{\Omega\times \R/2\pi\mathbb Z} f(t)\zeta(x)  d\sigma(x,t),
\end{equation}
 for every $\zeta\in C^\infty_c(\Omega) $ 
and every $f \in C^0(\R/2\pi\mathbb Z,\R)$. 
\end{lemma} 
\begin{proof}
We consider, for any fixed $\zeta\in C_c^\infty(\Omega)$, the linear functional $T_\zeta\colon C^0(\R/2\pi\Z,\R)\to\R$ given by
\begin{equation*}
T_\zeta(f)=\langle\nabla\cdot \Phi_f(m),\zeta\rangle.
\end{equation*}
Each functional $T_\zeta$ is continuous, since
\begin{align*}
\abs{T_\zeta(f)}&\leq\norm{\Phi_f}_\infty\norm{\nabla\zeta}_\infty \leq C\norm{\nabla\zeta}_\infty \norm{f}_\infty.
\end{align*}
On the other hand, for any 
$U\subset\subset\Omega$ and 
 $f\in C^0(\R/2\pi\Z,\R)$, by~\eqref{eq:wfep} it holds
\begin{equation*}
\abs{T_\zeta(f)}\leq \norm{{\nabla}\cdot\Phi_f(m)}_{\mathcal M(U)}\qquad\forall \zeta\in C_c^\infty(U),\;\norm{\zeta}_\infty\leq 1.
\end{equation*}
Applying Banach-Steinhaus' theorem we deduce  the existence of $C(U)>0$ such that
\begin{equation*}
\abs{\langle\nabla\cdot \Phi_f(m),\zeta\rangle}\leq C(U)\norm{f}_\infty\norm{\zeta}_\infty,
\end{equation*}
for all $f\in C^0(\R/2\pi\Z,\R)$ and $\zeta\in C_c^\infty(U)$. 
Since  tensor products are dense in 
$C^0_c(\Omega\times \R/2\pi\Z)$, 
by Riesz' representation theorem this implies the existence of $\sigma\in \mathcal M_{\loc}(\Omega\times\R/2\pi\Z)$ satisfying \eqref{eq:sigma}.
\end{proof}

By Lemma~\ref{l:linknuPhif} and \ref{l:sigma} above, and since by definition $f=\psi_f'$, we have
\begin{equation*}
\langle \nu-\partial_t\sigma,\psi_f(t)\zeta(x)\rangle =0\qquad\forall f\in C^\infty(\R/2\pi\Z,\R),\,\zeta\in C_c^\infty(\Omega).
\end{equation*}
However $\psi_f$ cannot be any arbitrary function $\psi\in C^\infty(\R/2\pi\Z,\R)$. In fact it holds
\begin{equation*}
\left\lbrace \psi_f\colon f\in C^\infty(\R/2\pi\Z)\right\rbrace =\left\lbrace \psi\in C^\infty(\R/2\pi\Z)\colon \psi(0)=0\text{ and }\int_0^{2\pi}\psi_f(s)e^{is}\,ds =0\right\rbrace.
\end{equation*}
In other words, we have thus far determined $\nu$ up to the Fourier modes $\lbrace 1,\cos t,\sin t\rbrace$ in the $t$-variable. The next lemma takes care of those modes.

\begin{lemma}
For all $\zeta\in C_c^\infty(\Omega)$ it holds
\begin{equation*}
\langle \nu,\zeta(x)\rangle = \langle \nu, \zeta(x)\cos t\rangle = \langle \nu, \zeta(x)\sin t\rangle =0.
\end{equation*}
\end{lemma}
\begin{proof}
We compute
\begin{align*}
\langle\nu,\zeta(x)\rangle&=-\int_\Omega\nabla\zeta(x)\cdot\int_{\R/2\pi\Z}\one_{e^{it}\cdot m>0}e^{it} \,dt\,dx\\
&=-\int_\Omega\nabla\zeta(x)\cdot (2 m(x)) \, dx \\
&=2\langle\nabla\cdot m,\zeta\rangle =0,\\
\langle\nu,\zeta(x)\cos t\rangle&=-\int_\Omega\nabla\zeta(x)\cdot\int_{\R/2\pi\Z}\cos t \one_{e^{it}\cdot m>0}e^{it} \,dt\,dx \\
&= \int_\Omega\nabla\zeta(x)
\cdot\left(\begin{array}{c}\frac{\pi}{2}\\ 0\end{array}\right)\,dx =0,
\end{align*}

and similarly $\langle \nu,\zeta(x)\sin t \rangle =0$.
\end{proof}

\begin{remark}
A similar computation shows that
\begin{equation*}
\langle \nu,\zeta(x)\cos((2k+1)t\rangle = \langle \nu,\zeta(x)\sin((2k+1)t\rangle =0\qquad\forall k\in\mathbb N.
\end{equation*}
Therefore the measure $\sigma$ does not have odd frequency Fourier modes. It can also be checked directly that for $f(t)=\cos((2k+1)t)$ and $f(t)=\sin((2k+1)t)$ it holds $\Phi_f\equiv 0$, which implies the same conclusion.
\end{remark}

\begin{proof}[Proof of Proposition~\ref{p:1}]
For $f(t)=\cos t$ or $f(t)=\sin t$ we have $\tilde f=0$ and therefore $\Phi_f=0$. By Lemma~\ref{l:sigma} this implies
\begin{equation*}
\langle \partial_t\sigma,\psi(t)\zeta(x)\rangle=0\quad\text{ for }\psi(t)=1\text{ or }\cos t\text{ or }\sin t.
\end{equation*}
We deduce that
\begin{equation*}
\langle \nu-\partial_t\sigma,\psi(t)\zeta(x)\rangle =0,
\end{equation*}
for any $\zeta\in C_c^\infty(\Omega)$ and $\psi\in C^\infty(\R/2\pi\Z)$, which proves \eqref{eq:kin}.
\end{proof}

\subsection{Kinetic formulation implies Besov regularity}

\begin{proposition}\label{p:2}
If $m$ satisfies the kinetic equation \eqref{eq:kin}, then it belongs to $B^{1/3}_{3,\infty;\loc}(\Omega)$.
\end{proposition}

The proof of Theorem~\ref{t:main} is inspired from the kinetic averaging lemma in 
\cite{golseperthame13} for 1D scalar conservation laws, and the way it is revisited in \cite{goldmanjosienotto15}. 
Following \cite{goldmanjosienotto15,crippaottowestdickenberg08} we make use of the 
following quantity to control spatial increments of $m$ at a fixed scale $h$ in the direction $e$. 
Let $m:\Omega\rightarrow S^1$ be measurable: given $h>0$, $|e|=1$ and $x\in\Omega$ we set
\begin{equation*}
\Delta (x,h,e) = \iint_{S^1\times S^1}\varphi(\xi,\eta)(\xi\wedge\eta) D_e^h\chi(x,\xi) D_e^h\chi(x,\eta) d\xi d\eta, 
\end{equation*}
where $\chi(x,\xi) =  \one_{\xi\cdot m(x)>0}$, $D_e^h\chi(x,\cdot)=D^{he}\chi(x,\cdot)=\chi(x+he,\cdot)-\chi(x,\cdot)$ 
 and
\begin{equation}\label{eq:phi}
\varphi(\xi,\eta)=(\one_{\xi\cdot\eta>0}-\one_{\xi\cdot\eta<0})(\one_{\xi\wedge\eta>0}-\one_{\xi\wedge\eta<0}).
\end{equation}

The next Lemma describes the coerciveness properties of the function $\Delta(x,h,e)$, with respect to the averaged quantities
\[
\frac 12\int_{S^1} \xi \chi(x+he,\xi) d\xi = m(x+he)\qquad\text{ and }\qquad \frac 12\int_{S^1} \xi \chi(x,\xi) d\xi = m(x). 
\]
\begin{lemma}\label{l:lemma2.1} 
Given $m\colon\Omega\rightarrow S^1$, $x\in\Omega$ and $0<h<dist(x,\partial\Omega)$, it holds: 
\begin{equation*}
\Delta(x,h,e) \gtrsim \abs{m(x+he)-m(x)}^3 = |D^h_em(x)|^3.
\end{equation*}
\end{lemma}
\begin{proof}
It holds $\Delta(x,h,e)=\Xi(m(x+he),m(x))$ where
\begin{equation*}
\Xi(m_1,m_2)=\iint_{S^1\times S^1}\varphi(\xi,\eta)(\xi\wedge\eta)\left(\one_{\xi\cdot m_1>0}-\one_{\xi\cdot m_2>0}\right)\left(\one_{\eta\cdot m_1>0}-\one_{\eta\cdot m_2>0}\right).
\end{equation*}
Therefore it suffices to prove that
\begin{equation*}
\Xi(m_1,m_2)\gtrsim \abs{m_1-m_2}^3\qquad\forall m_1,m_2\in S^1.
\end{equation*}
It is easily checked that $\Xi(m_1,m_2)=\Xi(m_2,m_1)$ and, since
\begin{equation}\label{eq:so2}
\varphi(R\xi,R\eta)=\varphi(\xi,\eta)\qquad\forall R\in SO(2),
\end{equation} 
that $\Xi(Rm_1,Rm_2)=\Xi(m_1,m_2)$ for all $R\in SO(2)$. Therefore it is enough to consider the case $m_1=e^{-i\beta}$, $m_2=e^{i\beta}$ for some $\beta\in [0,\pi/2]$ and to prove
\begin{equation*}
\Xi(e^{-i\beta},e^{i\beta})\gtrsim \beta^3\qquad\forall \beta\in [0,\pi/2].
\end{equation*}
The function $\varphi$ defined in  \eqref{eq:phi} that appears in the definition of $\Xi$ satisfies
\begin{align*}
\varphi(e^{i\theta},e^{i\psi})&=\widetilde\varphi(\psi-\theta),\qquad
\widetilde\varphi(\omega)  =\one_{\omega\in (0,\pi/2)\text{ mod }\pi}-\one_{\omega\in (\pi/2,\pi)\text{ mod }\pi}.
\end{align*}
We compute
\begin{align*}
\Xi(e^{-i\beta},e^{i\beta})
& = \int_{-\pi}^{\pi}\int_{-\pi}^{\pi}
\widetilde\varphi(\psi-\theta)\sin(\psi-\theta)\\
&\qquad\cdot\left(\one_{e^{i\theta}\cdot e^{-i\beta}>0}-\one_{e^{i\theta}\cdot e^{i\beta}>0}\right)
\left(\one_{e^{i\psi}\cdot e^{-i\beta}>0}-\one_{e^{i\psi}\cdot e^{i\beta}>0}\right)\, d\theta d\psi \\
& = \int_{-\pi}^{\pi}
\widetilde\varphi(\omega)\sin(\omega) \gamma(\omega) d\omega,\\
\text{where }\gamma(\omega)=&\int_{-\pi}^{\pi}\overline\chi(\theta)\overline\chi(\theta+\omega) \, d\theta,\\
\text{and }\overline\chi(\theta)&=\one_{e^{i\theta}\cdot e^{-i\beta}>0}-\one_{e^{i\theta}\cdot e^{i\beta}>0}.
\end{align*}
Note that $\overline\chi(\theta+\pi)=\overline\chi(-\theta)=-\overline\chi(\theta)$ for almost every $\theta\in\mathbb R$. Therefore $\omega\mapsto \widetilde\varphi(\omega)\sin(\omega)\gamma(\omega)$ is $\pi$-periodic and even, and
\begin{equation}\label{eq:Xigamma}
\Xi(e^{-i\beta},e^{i\beta})=4\int_{0}^{\pi/2}\widetilde\varphi(\omega)\sin(\omega)\gamma(\omega) d\omega.
\end{equation}
Moreover the integrand defining $\gamma$ is $\pi$-periodic in $\theta$, hence for all $\omega\in (0,\pi/2)$ we have
\begin{equation*}
\gamma(\omega)=2\int_0^{\pi}\overline\chi(\theta)\overline\chi(\theta+\omega)\, d\theta.
\end{equation*}
Assume first $\beta\in [0,\pi/4]$. Then for $\theta\in (0,\pi)$ it holds
\begin{equation*}
\overline\chi(\theta)\overline\chi(\theta+\omega)
=
\begin{cases}
\one_{\theta\in [\pi/2-\beta,\pi/2+\beta-\omega)} &\text{ if }\omega\in [0,2\beta],\\
0 &\text{ if }\omega\in [ 2\beta,\pi/2],\\\end{cases}
\end{equation*}
and we find
\begin{align*}
\gamma(\omega)&=
\begin{cases}
2\cdot(2\beta -\omega) & \text{ if } \omega\in [0,2\beta],\\ 
0 & \text{ if }\omega\in [2\beta,\pi/2].
\end{cases}\\
& =2\cdot (2\beta -\omega)_+\qquad\forall \omega\in [0,\pi/2],
\end{align*}
Plugging this into \eqref{eq:Xigamma} we deduce
\begin{align*}
\Xi(e^{-i\beta},e^{i\beta})
& =8\int_0^{2\beta}(2\beta-\omega)\sin\omega\, d\omega\\
&= 8\cdot(2\beta-\sin(2\beta))\gtrsim\beta^3,
\end{align*}
for all $\beta\in [0,\pi/4]$.

Consider now $\beta\in [\pi/4,\pi/2]$. For $\theta\in [0,\pi]$ we have
\begin{equation*}
\overline\chi(\theta)\overline\chi(\theta+\omega)
=
\begin{cases}
\one_{\theta\in [\pi/2-\beta,\pi/2+\beta-\omega)} &\text{ if }\omega\in [0,\pi-2\beta]\\
\one_{\theta\in [\pi/2-\beta,\pi/2+\beta-\omega)}
-\one_{\theta\in (3\pi/2-\beta-\omega,\pi/2+\beta)}&\text{ if }\omega\in [\pi-2\beta,\pi/2],
\end{cases}
\end{equation*}
and therefore
\begin{align*}
\gamma(\omega)&=
\begin{cases}
2\cdot(2\beta-\omega)&\text{ if }\omega\in [0,\pi-2\beta],\\
2\cdot (\pi-2\omega)&\text{ if }\omega\in [\pi-2\beta,\pi/2],
\end{cases}
\\
\Xi(e^{-i\beta},e^{i\beta})& = 8\int_0^{\pi-2\beta}(2\beta-\omega)\sin\omega\, d\omega + 8\int_{\pi-2\beta}^{\pi/2}(\pi-2\omega)\sin\omega\,d\omega \\
& =8\cdot(2\beta-(\pi-4\beta)\cos(2\beta)-\sin(2\beta)) \\
&\quad
+8\cdot (2\sin(2\beta) -2 + (\pi-4\beta)\cos(2\beta )) \\
& =8\cdot(\sin(2\beta) + 2\beta -2)\geq 
8\left(\frac{\pi}{2} -1\right)\gtrsim \beta^3,
\end{align*}
for all $\beta\in [\pi/4,\pi/2]$.
 \end{proof}

To obtain bounds for the integral of $\Delta(x,h,e)$ on $\Omega$, when $m$ satisfies the kinetic formulation  \eqref{eq:kin}, we use the following lemma, that 
estimates its derivative with respect to $h$:
\begin{lemma}\label{l:lemma2.2}
Suppose that $m$ satisfies the kinetic formulation of the eikonal equation \eqref{eq:kin}, and that 
$\Omega'\subset\subset
\Omega''\subset\subset \Omega$. 
We then have for all unit vectors $e$ and  $|h|\lesssim dist(\Omega',\partial\Omega'')$:
\begin{equation}\label{eq:derivativedelta}
\int_{\Omega'}\Delta(x,h,e)dx \lesssim |h|(1 + \|\sigma\|_{\mathcal M(\Omega''\times\R/2\pi\mathbb Z  )}),
\end{equation}
where the multiplicative constant depends on the distance between $\Omega'$ and $\partial\Omega''$.
\end{lemma}

\begin{proof}
Assume \eqref{eq:kin}:
\[
e^{is}\cdot	\nabla_x\one_{e^{is}\cdot	m(x)>0} = \partial_s\sigma \qquad\text{ in }\mathcal D '(\Omega\times\R/2\pi\mathbb Z  ),
\]
and let us assume to have intermediate domains $\widetilde\Omega, \Omega''$ with 
$\Omega'\subset\subset\widetilde\Omega\subset\subset\Omega''\subset\subset\Omega$
and such that the distances among the boundaries of the first three are comparable.
We perform the calculation of $\partial_h\int_{\Omega'} \Delta(x,h,e)dx$ for a regularized integrand, namely:
\begin{itemize}
\item we regularize the equation \eqref{eq:kin} by convolving with respect to $x$ with a smooth approximation of the identity $\rho_\e$:
\begin{equation}\label{eq:kinreg}
e^{is}\cdot	\nabla_x\chi_\e(x,e^{is}) = \partial_s\sigma_\e,\qquad  \chi_\e = 
\left(\one_{e^{is}\cdot	m(x)>0}\right)\underset{x}{*}\rho_\e,\quad\sigma_\e =\sigma\underset{x}{*}\rho_\e;
\end{equation}
here $\e<dist(\Omega,\partial\Omega'')$; 
\item we approximate $\varphi$ \eqref{eq:phi}  by a smooth $\varphi_\delta$. The calculations below are valid for a generic $\varphi$ and only use the skew-symmetry property $\varphi(\xi,\eta)=-\varphi(\eta,\xi)$. 
 Assuming in addition the $SO(2)$ invariance property \eqref{eq:so2}, and 
  parametrizing with the angle between $\xi$ and $\eta$, 
  these conditions amount to require that 
 $\widetilde\varphi\colon s\mapsto \varphi(1,e^{is}) $ is odd and $2\pi$ periodic. In turn, a convolution on the real line with a smooth even kernel, at scale $\delta$, preserves both these properties.   Explicitly, we set
\begin{equation*}
\varphi_\delta(e^{i\theta},e^{i\psi})=\widetilde\varphi_\delta(\psi-\theta),\qquad\widetilde\varphi_\delta=\widetilde\varphi * \rho_\delta,
\end{equation*}
for some smooth even kernel $\rho$. This approximation has the following properties:
\begin{equation}\label{eq:phidelta}
\widetilde\varphi_\delta\to\widetilde\varphi\text{ a.e.},\qquad
\abs{\widetilde\varphi_\delta}\leq 1,\qquad \norm{\widetilde\varphi_\delta'}_{L^1(\R/2\pi\Z)}\leq 8.
\end{equation} 

\end{itemize}

The explicit dependence of the function $\Delta$ on the parameters $\e,\delta$ is omitted in the first calculations.

We assume without loss of generality that $e = e_1$ and use the notations $\chi^h(x,\xi)=\chi(x+he_1,\xi)$ 
and  $D_1^h\chi=\chi^h-\chi$.  Let $x\in\tilde\Omega$ and $|h|<\dist(\tilde\Omega,\partial\Omega'')$. 
Using the skew-symmetry of $\varphi$, we have
\begin{equation*}
  \begin{split}
\frac{\partial}{\partial h}\Delta(x,e,h)&=\frac{\partial}{\partial h} \iint_{S^1\times S^1}\varphi(\xi,\eta)(\xi\wedge\eta) D_{1}^h\chi(x,\xi) D_{1}^h\chi(x,\eta) d\xi d\eta \\
&=  \iint_{S^1\times S^1}\varphi(\xi,\eta)(\xi\wedge\eta) [ \partial_1\chi^h(x,\xi) D^{h}_1\chi(x,\eta) +\partial_1\chi^h(x,\eta) D^{h}_1\chi(x,\xi)   ] d\xi d\eta \\
&= 2 \iint_{S^1\times S^1}\varphi(\xi,\eta)(\xi_1 \eta_2) [ \partial_1\chi^h(x,\xi) D^{h}_1\chi(x,\eta) +\partial_1\chi^h(x,\eta) D^{h}_1\chi(x,\xi)   ] d\xi d\eta. 
  \end{split}
\end{equation*}
Letting $\nu(x,e^{is}):=\partial_s\sigma(x,s)$, we use the
the equation \eqref{eq:kin} in the form 
\begin{equation}\label{eq:kinbis}
\xi_1\partial_1\chi(x,\xi) + \xi_2\partial_2\chi(x,\xi) =\nu(x,\xi),
\end{equation}
to replace $\xi_1\partial_1\chi^h(x,\xi)$ in the above and obtain
\begin{equation*}
  \begin{split}
\frac{\partial}{\partial h}\Delta(x,e,h)&= 2 \iint_{S^1\times S^1}\varphi(\xi,\eta) \eta_2 [(\nu^h(x,\xi)-\xi_2\partial_2\chi(x,\xi))D^{h}_1\chi(x,\eta) +\xi_1\partial_1\chi^h(x,\eta) D^{h}_1\chi(x,\xi)   ] d\xi d\eta \\
  &=  2 \iint_{S^1\times S^1}\varphi(\xi,\eta) \eta_2 [(\nu^h(x,\xi)-\xi_2\partial_2\chi(x,\xi))D^{h}_1\chi(x,\eta) -\xi_1\chi^h(x,\eta)\partial_1 D^{h}_1\chi(x,\xi)   ] d\xi d\eta \\
  &\quad + \partial_1\left[  2 \iint_{S^1\times S^1}\varphi(\xi,\eta) \eta_2 \xi_1\chi^h(x,\eta) D^{h}_1\chi(x,\xi)    d\xi d\eta \right] 
  =:I_1 + \partial_1 A_1.
    \end{split}
\end{equation*}
The term $\partial_1 A_1$ is a boundary term and will be treated at the end. 
Focusing on $I_1$, we can expand $\xi_1\partial_1 D^{h}_1\chi(x,\xi)$ and use \eqref{eq:kinbis} to deduce
\begin{equation*}
  \begin{split}
  I_1 &= 2 \iint_{S^1\times S^1}\varphi(\xi,\eta) \eta_2 \Big[(\nu^h(x,\xi)-\xi_2\partial_2\chi(x,\xi))(\chi^h(x,\eta)-\chi(x,\eta)) \\
  &\qquad  -   (\nu^h(x,\xi) - \nu(x,\xi) -\xi_2\partial_2\chi^h(x,\xi) +\xi_2\partial_2\chi(x,\xi))\chi^h(x,\eta)    \Big] d\xi d\eta \\
  &=2 \iint_{S^1\times S^1}\varphi(\xi,\eta) \eta_2 \Big[-\nu^h(x,\xi)\chi(x,\eta)  + \nu(x,\xi)\chi^h(x,\eta) \\
  &\qquad  +\xi_2\partial_2\chi^h(x,\xi)\chi(x,\eta) - \xi_2\partial_2\chi(x,\xi)\chi^h(x,\eta)
   \Big] d\xi d\eta \\
    &=2 \iint_{S^1\times S^1}\varphi(\xi,\eta) \eta_2 [-\nu^h(x,\xi)\chi(x,\eta)  + \nu(x,\xi)\chi^h(x,\eta) ] \\
    &\quad +2 \iint_{S^1\times S^1}\varphi(\xi,\eta) \eta_2 \xi_2[\partial_2\chi^h(x,\xi)\chi(x,\eta) - \partial_2\chi(x,\xi)\chi^h(x,\eta)]d\xi d\eta .\\
   \end{split}
\end{equation*}
Exchanging $\xi$ and $\eta$ only in the last term of the second integral, we can rewrite
\begin{equation*}
  \begin{split}
  I_1 
    &=2 \iint_{S^1\times S^1}\varphi(\xi,\eta) \eta_2 [-\nu^h(x,\xi)\chi(x,\eta)  + \nu(x,\xi)\chi^h(x,\eta) ] \\
    &\quad +2 \iint_{S^1\times S^1}\varphi(\xi,\eta) \eta_2 \xi_2[\partial_2\chi^h(x,\xi)\chi(x,\eta) + \partial_2\chi(x,\eta)\chi^h(x,\xi)]d\xi d\eta \\
    &=2 \iint_{S^1\times S^1}\varphi(\xi,\eta) \eta_2 [-\nu^h(x,\xi)\chi(x,\eta)  + \nu(x,\xi)\chi^h(x,\eta) ] \\
    &\quad +\partial_2\left[2 \iint_{S^1\times S^1}\varphi(\xi,\eta) \eta_2 \xi_2 \chi(x,\eta)\chi^h(x,\xi) d\xi d\eta\right] =: I_2 + \partial_2 A_2.
   \end{split}
\end{equation*}
Therefore we have $\frac{\partial}{\partial h}\Delta(x,h,e) = I_2 +  \partial_1 A_1+ \partial_2 A_2$, where 
\[
|A| = |(A_1,A_2)|\leq 8\pi\qquad \text{ pointwise }\forall (x,h,e).
\]
The most important term in the estimate is $I_2$, since the extra term is a divergence $\nabla_x\cdot A$ of a bounded vectorfield, hence 
it can be treated as a boundary term. 
In polar coordinates $I_2$ becomes:
\begin{equation*}
  \begin{split}
I_2 & = 2 \iint_{[0,2\pi[\times[0,2\pi[}
\widetilde\varphi(\psi-\theta) \sin\psi [-\partial_\theta\sigma^h(x,\theta) \chi(x,e^{i\psi})  + \partial_\theta\sigma(x,\theta)\chi^h(x,e^{i\psi}) ] 
d\theta d\psi \\
&= 2 \iint_{[0,2\pi[\times[0,2\pi[}\widetilde\varphi'(\psi-\theta) \sin\psi  [-\sigma^h(x,\theta) \chi(x,e^{i\psi})  + \sigma(x,\theta)\chi^h(x,e^{i\psi}) ] 
d\theta d\psi.
 \end{split}
\end{equation*}
Recall now that the above was derived for an approximation $\varphi_\delta$ of $\varphi$ and for a solution of the 
 regularized kinetic equation \eqref{eq:kinreg} at scale $\e$. Writing this dependence explicitly we have:
\begin{align*}
I_2 = I^{\e,\delta}_2
& = -2\int_0^{2\pi}\sigma_\e^h(x,e^{i\theta})\int_0^{2\pi}\widetilde\varphi_\delta'(\psi-\theta)\chi_\e(x,e^{i\psi})\sin\psi\,d\psi \, d\theta \\
&\quad + 2\int_0^{2\pi}\sigma_\e(x,e^{i\theta})\int_0^{2\pi}\widetilde\varphi_\delta'(\psi-\theta)\chi_\e^h(x,e^{i\psi})\sin\psi\,d\psi \, d\theta.
\end{align*}
Recalling \eqref{eq:phidelta} and the fact that $\abs{\chi_\e}\leq 1$ a.e., we deduce
\begin{equation*}
\abs{I_2^{\e,\delta}}\lesssim \int_0^{2\pi}\left(\abs{\sigma_\e^h(x,\theta)}+\abs{\sigma_\e(x,\theta)}\right)\,d\theta.
\end{equation*}
Plugging this estimate into the identity $\partial_h \Delta^{\e,\delta} = I_2^{\e,\delta} +  \nabla_x \cdot A^{\e,\delta}$ yields
\begin{equation*}
\frac{\partial}{\partial h}\Delta^{\e,\delta}(x,h,e)
\lesssim \int_0^{2\pi}\left(\abs{\sigma_\e^h(x,\theta)}+\abs{\sigma_\e(x,\theta)}\right)\,d\theta + \nabla_x \cdot A^{\e,\delta}.
\end{equation*}
Recalling that $A^{\e,\delta}$ is a uniformly bounded vector field, we may test the above against any nonnegative $\gamma \in C^\infty_c(\widetilde\Omega)$ and obtain
\begin{align*}
\frac{\partial}{\partial h} \int_\Omega \gamma(x)\Delta^{\e,\delta}(x,h,e)\, dx&\lesssim\norm{\gamma}_{C^0}\left(\norm{\sigma_\e^h}_{L^1(\widetilde\Omega\times\R/2\pi\Z)}+\norm{\sigma_\e}_{L^1(\widetilde\Omega\times\R/2\pi\Z)}\right) + \norm{\nabla\gamma}_{C^0}\\
&\lesssim \norm{\gamma}_{C^0}\norm{\sigma}_{\mathcal M(\Omega''\times\R/2\pi\Z)} +\norm{\nabla\gamma}_{C^0},
\end{align*}
for $\abs{h}+\e\leq \dist(\widetilde\Omega,\partial\Omega'')$ and $\delta>0$. Integrating with respect to $h$  we find that 
\begin{equation*}
\frac{1}{\abs{h}}\int_\Omega\gamma(x)\Delta^{\e,\delta}(x,h,e)\, dx \lesssim \norm{\gamma}_{C^0}\norm{\sigma}_{\mathcal M(\Omega''\times\R/2\pi\Z)} +\norm{\nabla\gamma}_{C^0}.
\end{equation*}
By dominated convergence we may pass to the limit $\e,\delta\to 0$ in the left-hand side. Then it remains to choose $\gamma\equiv 1$ in $\Omega'$  
to obtain the claimed estimate \eqref{eq:derivativedelta}.
\end{proof} 

We can now prove Proposition~\ref{p:2}: 
\begin{proof}
The proof follows combining the results of Lemmas \ref{l:lemma2.1} and \ref{l:lemma2.2}. For $t < \dist(\Omega',\partial\Omega'')$, 
\begin{align*}
 [N_{t}(m,\Omega')]^3 & =  \frac 1t \sup_{|e|=1,\abs{h}\leq t}\int_{\Omega'}|D^h_em(x)|^3dx \\
& \lesssim \frac 1t \sup_{|e|=1,\abs{h}\leq t} \int_{\Omega'}\Delta(x,e,h) dx \\
& \lesssim 1+\|\sigma\|_{\mathcal M(\Omega''\times \R/2\pi\mathbb Z)}.
\end{align*}
For other values of $t$ up to $1$, the triangular inequality and the boundedness of $m$ yield a trivial control on $N_t(f,\Omega)$.
Together, these estimates give the desired bound on the local Besov norm $B^{1/3}_{3,\infty}(\Omega')$. 
\end{proof}

\subsection{Besov regularity implies finite entropy}

For the proofs of the next propositions and lemmas, we let $m_\e := m*\rho_\e$ be a regularization with a standard kernel 
(with $\spt(\rho)\subset B_1$ and $\nabla\rho = 0$ in $B_{1/2}$). 

\begin{proposition}\label{p:3}
If $m$ solves \eqref{eq:m} and belongs to the space $B^{1/3}_{3,\infty,\loc}(\Omega)$, then $m$ has locally finite strong entropy production \eqref{eq:fep}:
\[ 
 \bigvee_{\Phi\in ENT, \|D^2\Phi\|_{\infty}\leq 1} \|{\nabla}\cdot\Phi(m)\| (A) \lesssim 
  [m]^3_{B^{\frac 13}_{3,\infty}(A)} \qquad\text{for } A\subset\subset\Omega.   
  \]
\end{proposition}
\begin{proof}
For a given $\Phi \in ENT$, we consider its extension $\widetilde\Phi\in C_c^2(\R^2,\R^2)$ given in polar coordinates by $\widetilde\Phi(re^{i\theta})=\eta(r)\Phi(e^{i\theta})$,
where $\eta$ is a fixed cut-off function $\eta\in C_c^\infty(0,\infty)$ satisfying $\eta\equiv 0$ outside $(1/2,2)$ and $\eta(1)=1$.

Following \cite{DMKO01,delellisignat15}, for the mollified field $m_\e$ we can single out in the entropy production the contribution of the radial oscillation:
\[
{\nabla}\cdot\widetilde\Phi(m_\e) = \Psi(m_\e)\cdot{\nabla}(1-|m_\e|^2),
\]
where $\Psi \in C^1_c(\R^2,\R^2)$ is a regular vectorfield. Given  a test function $\phi\in C^\infty_c(\Omega)$ we can integrate by parts
\[
\langle\nabla\cdot\widetilde\Phi(m_\e),\phi\rangle = -\int_\Omega \phi(x)\Psi(m_\e(x))\cdot{\nabla}(1-|m_\e(x)|^2)dx =: A_\e[\phi] + B_\e[\phi],
\]
 where
\begin{equation*}
\begin{split}
A_\e[\phi] &=  \int_\Omega\nabla\phi(x)\cdot\Psi(m_\e(x))(1-|m_\e(x)|^2)dx,\\
B_\e[\phi] &=  \int_\Omega \phi(x){\nabla}\cdot[\Psi(m_\e(x))](1-|m_\e(x)|^2)dx \\
&=\int_\Omega \phi(x) Tr[D\Psi(m_\e(x)){\nabla} m_\e(x)] (1-|m_\e(x)|^2)dx .
\end{split}
\end{equation*}

While $A_\e[\phi]\rightarrow 0$, trivially because $|m|=1$ almost everywhere, the second integral $B_\e[\phi]$ can be bounded by
\[
B_\e[\phi]\lesssim \|\phi\|_{L^\infty}\|D\Psi\|_{L^\infty}\int_{ \spt(\phi)} |\nabla m_\e(x)||1-|m_\e(x)|^2| dx.
\]
Since $3$ and $\frac 32$ are dual exponents, using Lemmas \ref{l:lemma3.1} and \ref{l:lemma3.2} below,
 the last integral can be bounded by 
 \[
\|1-|m_\e|^2\|_{L^{\frac 32}(\spt(\phi))} \|\nabla m_\e\|_{L^3(\spt(\phi))}\lesssim
N_{\e}(m,\spt(\phi))^3.  
 \]
Noting that $\abs{D\Psi}\lesssim\vert D^2\widetilde \Phi\vert$ and letting $\e\to 0$ we deduce that
\begin{equation*}
\abs{\langle\nabla\cdot\Phi(m),\phi\rangle}\lesssim \norm{\phi}_{L^\infty}\norm{D^2\Phi}_{\infty}\liminf_{\e\to 0} N_\e(m,\spt(\phi))^3,
\end{equation*}
and therefore
\begin{equation*}
\norm{\nabla\cdot\Phi(m)}(U)\lesssim \norm{D^2\Phi}_{\infty}\liminf_{\e\to 0} N_\e(m,\overline U)^3,
\end{equation*}
for all $U\subset\subset\Omega$.
  Note that $N_\e(m,\overline U)$ involves integrals with respect to $x$ over the sets $\overline U$ and $\overline U+\e y$, hence given a finite family of open and distant sets $U_1,\dots,U_k\subset\subset A\subset\subset \Omega$, and 
 a corresponding family of entropies $\Phi_1,\dots,\Phi_k$ with $\norm{D^2\Phi_j}_\infty\leq 1$, if $\e$ is small enough it holds
 \[
 \sum_j \|\nabla\cdot\Phi_j(m)\|(U_j) \lesssim \liminf_{\e\to 0} \sum_j N_{\e}(m,U_j)^3 
 = \liminf_{\e\to 0} N_{\e}(m,A)^3  .
 \]
Recalling the definitions of the least upper bound measure and of the Besov seminorm, this implies the conclusion of Proposition~\ref{p:3}.
\end{proof}

In the proof of Proposition \ref{p:3}, we used the two following lemmas on the growth of certains norms of the regularized field $m_\e$. 
Their proof is an adaptation to the Besov scale of corresponding statements for Sobolev functions, treated in \cite{delellisignat15}. 
\begin{lemma}\label{l:lemma3.1}
If $m \in B^{1/3}_{3,\infty,\loc}(\Omega)$, and  $\Omega'\subset\subset\Omega$,  
then for every $\e\lesssim dist (\Omega', \partial \Omega)$ 
\begin{equation*} 
\int_{	\Omega'}  |\nabla m_\e|^3 dx\lesssim \e^{-2} N_{\e}(m,\Omega')^3.
\end{equation*}
\end{lemma}
\begin{proof}
As in \cite[Proof of Proposition 3, Step 6(ii)]{delellisignat15}, for $\e$ small enough we have the pointwise bound:
\[
|\nabla m_\e(x)|\leq \frac{\|\nabla \rho\|_\infty}{\e^3}\int_{B_\e(0)\setminus B_{\e/2}(0)}|m(x+z)-m(x)|dz .
\] 
Applying Jensen inequality and integrating in $\Omega'$ one obtains
\begin{align*}
\int_{\Omega'}|\nabla m_\e(x)|^3  dx &\lesssim 
\frac{1}{\e^5}
\int_{\Omega'}\int_{B_\e(0)\setminus B_{\e/2}(0)}{|m(x+z)-m(x)|^3}dz dx \\
& = \frac{1}{\e^3}
\int_{\Omega'}\fint_{B_\e(0)}{|D^z m(x)|^3}dz dx\\
&\leq \frac{1}{\e^2} N_{\e}(m,\Omega')^3.
\end{align*}
\end{proof}

\begin{lemma}\label{l:lemma3.2}
With the notations of Lemma~\ref{l:lemma3.1}, it holds:
\begin{equation*}
\int_{	\Omega'} (1- | m_\e|^2)^{3/2} dx\lesssim \e N_{\e}(m,\Omega')^3.
\end{equation*}
\end{lemma}
\begin{proof}
As in \cite[Proof of Proposition~3, Step~6(i)]{delellisignat15}, using that $|m|=1$ almost everywhere we obtain the pointwise bound
\begin{multline*}
|1-|m_\e|^2|(x) \lesssim \int_{B_\e(x)}\int_{B_\e(x)} \rho_\e(x-y)\rho_\e(x-z)|m(y)-m(z)|^2dydz \\
\lesssim  
\int_{B_\e(x)} \rho_\e(x-y)|m(y)-m(x)|^2dy
= \int_{B_\e(0)} \rho_\e(z)|m(x+z)-m(x)|^2dz
\end{multline*}
Then by H\"older's inequality
\begin{align*}
\int_{\Omega'}|(1-|m_\e(x)|^2)|^\frac{3}{2}dx & \lesssim 
\int_{\Omega'}  \int_{B_\e(0)} \rho_\e(z)|D^z m(x)|^3dz dx \\
& \leq  \sup_{|z|\leq \e} \| D^z m\|^3_{L^3(\Omega')}\\
&\leq \e N_{\e}(m,\Omega')^3.
\end{align*}
\end{proof}

To conclude the proof of Theorem~\ref{t:main} we remark that the implication (iv)$\Rightarrow$(i) is trivial.

\section{Corollaries and further comments}\label{s:cor}

\subsection{Sharp differentiability for zero energy states}
We observe that if $m\in B^{1/3}_{3,c_0,{\rm loc}}(\Omega)$ (in particular when $m\in B^{1/3}_{3,q,{\rm loc}}(\Omega)$, $q<\infty$), 
 then thanks to Lemmas~\ref{l:lemma3.1} and \ref{l:lemma3.2} we have 
\[
\e^{\frac 23}\nabla m_\e\rightarrow 0 \quad\text{in}\quad L^3_{{\rm loc}}(\Omega)\text{ and }\quad \e^{-\frac 23}(1-|m_\e|^2)\rightarrow 0\quad\text{in}\quad L^{\frac 32}_{{\rm loc}}(\Omega).
\]
Therefore the conclusion of Proposition~\ref{p:3} can be refined to 
\[
\|{\nabla}\cdot\Phi(m)\|=0\qquad\text{ for every }\Phi\in ENT.
\]
That is, slightly better regularity rules out entropy production. This in turn implies much stronger regularity properties: $ m$ is locally Lipschitz outside a locally finite set of vortices, \cite{jabinottoperthame02}.

\subsection{The mass of the entropy measure $\|\sigma\|(\Omega)$}
In the micromagnetics model studied by Rivi\`ere-Serfaty in \cite{riviereserfaty03}, twice the total variation of the kinetic measure provides a sharp asymptotic lower bound 
for the energy, \cite[Theorem 1]{riviereserfaty03}.
In this paragraph we investigate whether this property holds for our model \eqref{eq:m}, at least in the $BV$ case. Recall \cite{contidelellis07} that for $m=\nabla^\perp u\in BV(\Omega)$ satisfying \eqref{eq:m} it holds
\begin{equation*}
\big(\Gamma-\lim AG_\e\big)(u)=\frac 13 \int_{J_m}\abs{m^+-m^-}^3 d\mathcal H^1.
\end{equation*}
Hence the question we raise is whether this equals $2\norm{\sigma}(\Omega)$.

For a given $m\in BV(\Omega)$ satisfying \eqref{eq:m} we compute $\norm{\sigma}$ as follows.
In light of Lemma~\ref{l:sigma} it holds
\begin{equation*}
\|\sigma\| =\bigvee_{|f|\leq 1}\|\nabla\cdot\Phi_f(m)\|.
\end{equation*}
On the other hand, Remark~\ref{r:frames} and the results in \cite{ambdelman99} ensure that
\begin{equation}\label{eq:masssigma}
\begin{aligned}
2\norm{\sigma} & =
2 \bigvee_{|f|\leq 1}\|\nabla\cdot\Phi_f(m)\|\\
& \geq  \bigvee_{(\alpha_1,\alpha_2)}\|\nabla\cdot\Sigma_{\alpha_1,\alpha_2}(m)\| = \frac 13 |m^+ - m^-|^3\mathcal H^1\res J_{m}.
 \end{aligned}
\end{equation}

\begin{proposition}
If $D m$ has a nontrivial jump part, then the inequality in \eqref{eq:masssigma} is strict. 
\end{proposition}
\begin{proof}
According to \cite[Theorem~3]{ignatmerlet12}, since the set $\lbrace \Phi_f \colon\abs{f}\leq 1\rbrace$ is symmetric (stable under multiplication by $-1$) and equivariant (stable under conjugation by any rotation), it holds
\begin{equation}\label{eq:entropycost}
\bigvee_{|f|\leq 1}\|\nabla\cdot\Phi_f(m)\| = c(\abs{m^+-m^-})\mathcal H^1\res J_m.
\end{equation}
for a certain cost function $c$. This cost function is given by
\begin{equation*}
c(s)=\sup\left\lbrace \left(\Phi_f(m^+)-\Phi_f(m^-)\right)\cdot\nu\right\rbrace,
\end{equation*}
where the supremum is taken among :
\begin{itemize}
\item all possible jumps $m^\pm\in S^1$ of size $\abs{m^+-m^-}=s$, 
\item all possible normal vectors $\nu\in S^1$ with the admissibility condition $(m^+-m^-)\cdot\nu=0$ (due to the divergence constraint $\nabla\cdot m=0$),
\item  and all possible $f$ with $\abs{f}\leq 1$.
\end{itemize}
Using again the symmetry and equivariance of $\lbrace \Phi_f\rbrace$, we can simplify this as
\begin{equation*}
c(s)=\sup\left\lbrace \left(\Phi_f(e^{i\beta})-\Phi_f(e^{-i\beta})\right)\cdot e_1,\: \abs{f}\leq 1\right\rbrace\quad\text{for }s=2\sin\beta,\:\beta\in [0,\pi/2].
\end{equation*}
For  angles  $\beta\in [-\pi/2,\pi/2]$ (we want to apply this computation also to $-\beta$) it holds
\begin{align*}
e_1\cdot\Phi_f(e^{i\beta})&=\mathfrak{Re}(-i\varphi_f(\beta-\pi/2)+i\varphi_f(\beta+\pi/2))\\
&=\mathfrak{Re}\left(\int_0^{\beta-\pi/2}\psi_f(s)e^{is}\,ds -\int_0^{\beta+\pi/2}\psi_f(s)e^{is}\,ds\right)\\
&=-\int_{\beta-\pi/2}^{\beta+\pi/2}\psi_f(s)\cos s\, ds
=-\left[ \psi_f\sin\right]_{\beta-\pi/2}^{\beta+\pi/2}+\int_{\beta-\pi/2}^{\beta+\pi/2}\tilde f(s)\sin s\, ds\\
&=-\cos\beta \left(\int_0^{\beta+\pi/2}\tilde f + \int_0^{\beta-\pi/2}\tilde f \right) +\int_{\beta-\pi/2}^{\beta+\pi/2}\tilde f \sin \\
& = -\cos\beta\Bigg[ \int_0^{\beta+\pi/2}f -\frac{1}{2\pi}(\beta+\pi/2)\int_0^{2\pi}f -\frac{1}{\pi}\left(\int_0^{\beta+\pi/2}\cos\right)\int_0^{2\pi}f\cos \\
&\qquad\quad\qquad - \frac{1}{\pi}\left(\int_0^{\beta+\pi/2}\sin\right)\int_0^{2\pi}f\sin + \int_0^{\beta-\pi/2} f -\frac{1}{2\pi}(\beta-\pi/2)\int_0^{2\pi} f \\
&\quad\qquad\qquad -\frac{1}{\pi}\left(\int_0^{\beta-\pi/2}\cos\right)\int_0^{2\pi} f\cos -\frac{1}{\pi}\left(\int_0^{\beta-\pi/2}\sin\right)\int_0^{2\pi} f\sin \Bigg] \\
&\quad +\int_{\beta-\pi/2}^{\beta+\pi/2}f\sin - \frac{1}{2\pi}\left(\int_{\beta-\pi/2}^{\beta+\pi/2}\sin\right)\int_0^{2\pi} f\\
&\quad -\frac{1}{\pi}\left(\int_{\beta-\pi/2}^{\beta+\pi/2}\cos\sin\right)\int_0^{2\pi} f\cos-\frac{1}{\pi}\left(\int_{\beta-\pi/2}^{\beta+\pi/2}\sin^2\right)\int_0^{2\pi} f\sin \\
&=-\cos\beta\left[
\int_0^{\beta+\pi/2}f+\int_0^{\beta-\pi/2}f -\frac{\beta}{\pi}\int_0^{2\pi}f -\frac{2}{\pi}\int_0^{2\pi}f\sin\right]\\
&\quad
+\int_{\beta-\pi/2}^{\beta+\pi/2}f\sin -\frac{1}{\pi}\sin\beta \int_0^{2\pi}f -\frac 12\int_0^{2\pi} f\sin.
\end{align*}
Hence for any $\beta\in [0,\pi/2]$,
\begin{align*}
e_1\cdot\left(\Phi_f(e^{i\beta})-\Phi_f(e^{-i\beta})\right)
&=\cos\beta\left[
\int_0^{-\beta+\pi/2} f + \int_0^{-\beta-\pi/2}f -\int_0^{\beta+\pi/2}f -\int_0^{\beta-\pi/2} f + \frac{2\beta}{\pi}\int_0^{2\pi} f
\right]\\
&\quad + \int_{\beta-\pi/2}^{\beta+\pi/2}f\sin -\int_{-\beta-\pi/2}^{-\beta+\pi/2}f\sin -\frac 2\pi \sin\beta\int_0^{2\pi} f \\
& =-\cos\beta\int_{-\beta-\pi/2}^{\beta-\pi/2}f -\cos\beta\int_{-\beta+\pi/2}^{\beta+\pi/2} f + \int_{-\beta+\pi/2}^{\beta+\pi/2}f\sin -\int_{-\beta-\pi/2}^{\beta-\pi/2}f\sin \\
&\quad -\frac{2}{\pi}(\sin\beta-\beta\cos\beta)\int_0^{2\pi}f \\
&=\int_{0}^{2\pi} g_\beta f,
\end{align*}
where $g_\beta$ is $\pi$-periodic and
\begin{align*}
g_\beta(t)&=(\sin t-\cos\beta)\one_{\pi/2-\beta\leq t\leq \pi/2+\beta} - \frac{2}{\pi}(\sin\beta-\beta\cos\beta)\qquad\forall t\in [0,\pi].
\end{align*}
The above computation with $f(t)=\cos(2t)$ yields an entropy production equal to $(2\sin\beta)^3/6$, as expected. On the other hand the supremum of the above quantity over  $\abs{f}\leq 1$ is given by $\norm{g_\beta}_{L^1(0,2\pi)}$. This supremum is not attained  by a continuous function when $\beta>0$. 
In other words, for any jump of size $s>0$ we have $c(s)>s^3/6$. In view of \eqref{eq:entropycost} this shows that equality in \eqref{eq:masssigma} can not happen unless $Dm$ has a trivial jump part.
\end{proof}

To calculate the value of $c(s)$ we observe the following. 
Since $g_\beta$ is $\pi$-periodic and even it holds
\begin{equation*}
\norm{g_\beta}_{L^1(0,2\pi)}=4\int_0^{\pi/2}\abs{g_\beta}.
\end{equation*}
The function $g_\beta$  is negative in $[0,t_\beta)$ and positive in $(t_\beta,\pi/2]$, where $t_\beta\in [\pi/2-\beta,\pi/2]$ is characterized by
\begin{equation*}
\sin t_\beta -\cos\beta = \frac{2}{\pi}(\sin\beta-\beta\cos\beta).
\end{equation*}
Moreover it holds that $\int_0^{\pi/2} g_\beta =0$, hence we find
\begin{align*}
\int_0^{\pi/2}\abs{g_\beta}&=\int_0^{t_\beta}(-g_\beta) + \int_{t_\beta}^{\pi/2} g_\beta =2\int_{t_\beta}^{\pi/2} g_\beta \\
& = 2\cos t_\beta -2(\pi/2- t_\beta)\left(\cos\beta + \frac 2\pi (\sin\beta-\beta\cos\beta)\right).
\end{align*}
With this expression it can be checked that 
\begin{equation*}
\norm{g_\beta}_{L^1(0,2\pi)}\sim \frac 16 (2\beta)^3\qquad\text{as }\beta\to 0,
\end{equation*}
hence $c(s)\sim s^3/6$ for $s\to 0$, so that the measure $\norm{\sigma}$ does behave like the right-hand side of \eqref{eq:masssigma} for very small jumps.

\subsection{Partial regularity obtained by using only the Jin-Kohn entropies}

In this paragraph, we show how to obtain fractional differentiability of a solution $m$ of \eqref{eq:m} having 
finite entropy production for every entropy \eqref{eq:jk} in the class of Jin-Kohn: 
\[
\Sigma_{\alpha_1,\alpha_2}(z) = \frac 43\left((z\cdot\alpha_2)^3\alpha_1 + (z\cdot\alpha_1)^3\alpha_2\right). 
\]
Recall that $(\alpha_1,\alpha_2)$ is a positive orthonormal frame $(R_\theta e_1,R_\theta e_2)$, 
and notice moreover that every entropy is a linear combination of two basic entropies
 $\Sigma_{e_1,e_2}$ and $\Sigma_{\e_1,\e_2}$:
 \[
 \Sigma_{R_\theta e_1,R_\theta e_2}(z)=  \cos(2\theta)\Sigma_{e_1,e_2}(z) + \sin(2\theta) \Sigma_{\e_1,\e_2}(z),
 \]
 where $\e_1 = \tfrac{e_1 + e_2}{\sqrt{2}}$, $\e_1 = \tfrac{-e_1 + e_2}{\sqrt{2}}$. 
 
In \cite{lorentpeng17} the authors show that whenever the entropy production associated to $\Sigma_{e_1,e_2}$ and $\Sigma_{\e_1,\e_2}$ vanish (which is equivalent to all the Jin-Kohn entropy productions vanishing), then in fact all entropy productions vanish and the rigidity result of \cite{jabinottoperthame02} applies.  Hence it is natural to wonder whether in general, controlling the total variation of these two basic entropy productions is enough to obtain the $B^{1/3}_{3,\infty}$ estimate (which we obtained here using all entropy productions). We do not provide an answer to this question, but show how a method described in \cite{golse10} can be combined with estimates derived in \cite{lorentpeng17} to obtain a $B^{s}_{4,\infty}$ estimate for all $s<1/4$.

To this end we set
\begin{align*}
\Delta_{JK} (x,h,e)&= 
D^h_e \Sigma_{e_1,e_2}(m(x))
 \wedge
D^h_e \Sigma_{\e_1,\e_2}(m(x))\\
& = \det
D^h_e 
\left(
\Sigma_{e_1,e_2}(m), \Sigma_{\e_1,\e_2}(m)
\right)(x).
\end{align*}
Here we recall that $D^h_e$ denotes the spatial increment of size $h$ in direction $e$, that is $D^h_e f(x)=f(x+he)-f(x)$.
In \cite{lorentpeng17} the authors study some properties of the set $K$ of $2\times 2$ matrices given by
\begin{equation*}
K=\left\lbrace \left(\Sigma_{e_1,e_2}(m),\Sigma_{\e_1,\e_2}(m)\right)\colon m\in S^1\right\rbrace \subset \R^{2\times 2}.
\end{equation*}
One of its key properties, obtained in \cite[Lemma 7]{lorentpeng17} and inspired 
from the work of \v{S}verak on the Tartar conjecture \cite{sverak93}, is the following inequality:
\[
  \det( X- Y) \gtrsim |X-Y|^4\qquad\forall  (X,Y) \in  K\times K  . 
\]
Therefore the quantity $\Delta_{JK}$ can be estimated from below by  
\[
\Delta_{JK} (x,h,e) \gtrsim \abs{D_e^h\left(
\Sigma_{e_1,e_2}(m), \Sigma_{\e_1,\e_2}(m)
\right)(x)}^4\gtrsim |D^h_e m(x)|^4,
\]
where the last inequality follows from the (easily checkable) fact that $m\mapsto (\Sigma_{e_1,e_2},\Sigma_{\e_1,\e_2})$ is an immersion. 

Following \cite{golse10}, we aim to apply the div-curl Lemma, taking advantage of the fact that 
\[
\nabla\cdot \Sigma_{e_1,e_2}(m) = \mu_{e_1,e_2},\qquad \nabla\cdot \Sigma_{\e_1,\e_2}(m) = \mu_{\e_1,\e_2},
\]
are locally finite measures.
To this end let us fix $\chi$ a smooth cutoff function and set
\[
E:=\chi D^h_e  \Sigma_{e_1,e_2}(m), \qquad B:=\chi D^h_e  \Sigma_{\e_1,\e_2}(m). 
\]
\begin{lemma}\label{l:divcurl} For every $p\in ]1,\infty[$ the following estimate holds true:
\begin{equation}\label{eq:divcurl}
\int_{\R^2} E\wedge B \,dx \lesssim pp' 
\|E \|_{L^p} \|\nabla\cdot B \|_{W^{-1,p'}} +
\|B \|_{L^p} \|\nabla\cdot  E \|_{W^{-1,p'}} .
\end{equation}
\end{lemma}
\begin{proof}
The proof is nowadays standard, and we report it for the reader's convenience: for $1<p<\infty$, 
using the potential theoretic solution $\phi$ to $\Delta\phi = \nabla\cdot E$ , we find that
$E$ can be Hodge-decomposed as
\[
E=\nabla\phi+\nabla^\perp\psi,\qquad \|\nabla\phi\|_{L^{p'}(\R^2)} \lesssim  pp' \|\nabla \cdot E \|_{W^{-1,p'}(\R^2)}, 
\]
(\cite[Theorem 4.4.1]{grafakosclassical}, \cite{yudovich95}), which yields
\begin{align*}
\left| \int_{\R^2} E\wedge B \,dx \right| 
&\leq
\left|\int\nabla\phi\wedge B\right| + \left| \int\nabla^\perp\psi\wedge B\right| \\
&\lesssim
\|\nabla\phi\|_{L^{p'}} \|B \|_{L^p} + \left| \int\left(\psi - \fint_{\spt(B)}\psi\right)  \nabla\cdot B\right|  \\
&\lesssim
pp' \|\nabla \cdot E \|_{W^{-1,p'}(\R^2)} \|B \|_{L^p} +   \|\nabla\psi\|_{L^p}   \|\nabla \cdot B \|_{W^{-1,p'}(\R^2)}\\
&\lesssim 
pp' \|\nabla \cdot E \|_{W^{-1,p'}(\R^2)} \|B \|_{L^p} +  (1+pp') \|E\|_{L^p}  \|\nabla \cdot B \|_{W^{-1,p'}(\R^2)}. 
\end{align*}
The conclusion follows from $pp'\geq 4$.
\end{proof}
\begin{proposition}
Any solution $m$ to \eqref{eq:m} such that $\nabla\cdot\Sigma_{\alpha_1,\alpha_2}(m)\in\mathcal M_\loc(\Omega)$ for $(\alpha_1,\alpha_2)=(e_1,e_2)$ and $(\e_1,\e_2)$, belongs to 
$B^{s}_{4,\infty;\loc}(\Omega)$ for every $s<4$.
\end{proposition}
\begin{proof}
The  div-curl estimate of Lemma~\ref{l:divcurl} reduces the control of $\int \chi^2 |D^h_e m|^4$ to the estimate of the product
\[
\|\chi D^h_e  \Sigma_{e_1,e_2}(m) \|_{L^p} \|\nabla\cdot (\chi D^h_e  \Sigma_{\e_1,\e_2}(m)) \|_{W^{-1,p'}}
\]
and its companion obtained by exchanging $E$ and $B$. 
Let us for simplicity drop the frame index and write $\Sigma$ instead of $\Sigma_{\alpha_1,\alpha_1}$, and also write $D^h$ instead of $D^h_e$.  We start by estimating the $W^{-1,p'}$ norm of the second factor
\[
\nabla\cdot (\chi D^h_e  \Sigma (m))=
\chi D^h\mu + D^h_e\Sigma(m)\cdot\nabla\chi.
\]
By Sobolev embedding 
$W^{1,p}\subset C^{1-\frac 2p}$ for $p>2$, it holds
\begin{align*}
\|\chi D^h\mu\| _{W^{-1,p'}}
& = \sup\left\{
\int D^h(\chi\psi)d\mu, \|\psi\|_{W^{1,p}}\leq 1
\right\} \\
&\leq 
 \sup_{\|\psi\|_{W^{1,p}}\leq 1}
\|\mu\|_{\mathcal{M}} \|D^h(\chi\psi)\|_{\infty} \\
&\lesssim
 \|\mu\|_{\mathcal{M}} |h|^{1-\frac{2}{p}},
\end{align*}
and therefore
\[
\| 
\nabla\cdot (\chi D^h_e  \Sigma (m))
\|_{W^{-1,p'}} \lesssim \|\mu\|_{\mathcal{M}} |h|^{1-\frac{2}{p}}
+|h|.
\]
The $L^{p}$ norms $\|\chi D^h_e  \Sigma  \|_{L^p} $
 are uniformly bounded. 
Inserting this estimate in \eqref{eq:divcurl},  and choosing $p=-\log({|h|})$ \cite{yudovich95}, one easily obtains the modulus of continuity
\[
\int \chi^2 |D^h_e m|^4 dx\lesssim \int E\wedge B\, dx \lesssim (1+ \|\mu\|_{\mathcal{M}})|h|\log(\frac{1}{|h|})
\]
for  $|h|<\exp(-2)$. This implies $m \in B^{s}_{4,\infty}(\spt\chi)$ for every $s<\frac 14$.  
\end{proof}
\begin{remark}
It is unclear to the authors whether the $\tfrac 14$ exponent is optimal or not. 
\end{remark}

\bibliographystyle{acm}
\bibliography{ref}

\end{document}